\documentclass[psamsfonts]{amsart}
\usepackage{amssymb,amsfonts}
\usepackage[all,arc]{xy}
\usepackage{enumerate}
\usepackage{mathrsfs}
\usepackage{pstricks}
\usepackage{pst-node}

\usepackage{times}
\usepackage[T1]{fontenc}
\usepackage{mathrsfs}
\usepackage{epsfig}
\usepackage{color}

\newtheorem{theorem}{Theorem}[section]
\newtheorem{cor}[theorem]{Corollary}

\newtheorem{conjecture}[theorem]{Conjecture}

\theoremstyle{definition}
\newtheorem{definition}[theorem]{Definition}

\theoremstyle{remark}

\numberwithin{equation}{section}

\title{A note on the equivariant formal group law of the equivariant complex cobordism ring}

\author{William C. Abram}

\address{Department of Mathematics, Hillsdale College,
Hillsdale, MI 49242 USA}
\email{wabram@hillsdale.edu}

\date{March 26, 2015}
\thanks{
This work was supported by an NSF Graduate Research Fellowship.
}
\begin{document}

\begin{abstract}
For a finite abelian group $G$, we give a computation of the $G$-equivariant formal group law corresonding to the $G$-equivariant complex cobordism spectrum $MU_G$ with its canonical complex orientation.
\end{abstract}

\maketitle


%
%
%

\section{Introduction}

The goal of the present paper is to provide an algebraic description of the equivariant formal group law corresponding to the equivariant complex cobordism ring $MU_G$ for a finite abelian group $G$, which is given as Theorem ~\ref{maintheorem}, and is adapted from the author's doctoral thesis \cite{Abram13}. The paper \cite{AKprep} of the author and Kriz provides the corresponding algebraic computation of the coefficient ring $(MU_G)_*$ which facilitates this work. As in \cite{AKprep}, our description is in terms of pullback diagrams. The computation of \cite{AKprep} follows several other papers which contribute to the present algebraic understanding of the ring $(MU_G)_*$, namely the computations of Sinha \cite{Sinha01}, Strickland \cite{Strickland01}, and Kriz \cite{Kriz99}.

One motivation for pursuing this computation is the conjecture of Greenlees \cite[Conjecture 2.4]{Greenlees01} that the coefficient ring of the equivariant complex cobordism ring classifies equivariant formal group laws in the same way that non-equivariant complex cobordism classifies non-equivariant formal group laws, vis \'a vis Quillen's Theorem \cite{Quillen69} (c.f.e. \cite[Theorem 1.3.2]{Ravenel04}). This is  will be made precise in Section ~\ref{efgl} when we state the definition of an equivariant formal group law. Greenlees \cite{Greenlees01} shows that the equivariant complex cobordism ring classifies equivariant formal group laws over Noetherian rings, but the general conjecture is still open. The algebraic description of $(MU_G)_*$ in \cite{AKprep} and the corresponding description here of its equivariant formal group law are intended to partially illuminate the questions posed by the paper \cite{Greenlees01} of Greenlees.

In Section  ~\ref{cobordism} below we recall the computation of $(MU_G)_*$ given in \cite{AKprep}.
In Section ~\ref{efgl} we define equivariant formal group laws and provide the relevant background. 
In Section ~\ref{efglgeneralcase} we compute the equivariant formal group law of $MU_G$ for $G$ a finite abelian group. The main result of this paper is Theorem ~\ref{maintheorem}. To illuminate the description, in Section ~\ref{zpnefgl} we treat the example case $MU_{\mathbb{Z}/p^n}$ of a finite cyclic $p$-group.

\subsection*{Acknowledgements} The author would like to thank Igor Kriz for his mentorship on this project, and Kyle Ormsby, who read this work carefully as it appeared in the author's thesis. The author is also grateful to Andrew Blumberg, Jeanne Duflot, Dan Freed, Paul Goerss, Nitu Kitchloo, Tyler Lawson, J. Peter May, Jack Morava, and others for enlightening discussions about the equivariant complex cobordism ring. This work was supported by a National Science Foundation Graduate Research Fellowship, and much of the work was conducted at the University of Michigan, whose support the author acknowledges. 

\section{The equivariant complex cobordism ring of a finite abelian group} \label{cobordism}

Throughout this paper, fix a finite abelian group $G$. We recall now the computation of $(MU_G)_*$ given in \cite{AKprep}, which is given in that paper as the combination of Theorem 1 and Theorem 2. First we need some definitions. Let 
\begin{equation} P(G) = \{ \{ H_1 \subsetneq H_2 \subsetneq \cdots \subsetneq H_k\} | H_i \subset G \text{ is a subgroup for every } i\}.
\end{equation}
That is, elements of $P(G)$ are increasing chains of subgroups of $G$. Now suppose $S = \{ H_1 \subsetneq H_2 \subsetneq \cdots \subsetneq H_k\} \in P(G)$, and let $H_0 = \{e\}$, $H_{k+1} = G$. For $0 \leq j \leq k$, let $R_j$ be $G/H_j$-representatives of the nontrivial complex $H_{j+1}/H_j$ representations. We associate to the chain of subgroups $S$ a ring
\begin{align*} A_S = MU_*[u_L&,u_M^{-1},u_N^{(i)} | i > 0, L \in R_0 \coprod \cdots \coprod R_k, \\ &M \in R_0 \coprod \cdots \coprod R_{k-1}, N \in R_0],
\end{align*}
where of course $MU_*$ is the nonequivariant complex cobordism ring, whose algebraic structure was computed by Milnor \cite{Milnor60}.

We can topologize $A_S$ as follows. Say that a sequence of monomials from $A_S$
\[ \left< a_t \prod_{L \in R_1 \coprod \cdots \coprod R_k} u_L^{n(L,t)}  \right>_{t=1}^\infty
\]
with $0 \neq a_t \in MU_*[u_L^{\pm 1}, u_L^{(i)} |i > 0, L \in R_0]$ converges to $0$ if there is a $j \in \{1, \ldots, k\}$ such that $n(L,t)$ is eventually constant in $t$ when $L \in R_i$ and $i> j$, and 
\[ n(L,t) \rightarrow \infty \text{ as } t \rightarrow \infty
\]
for $L \in R_j$. A sequence $\left< p_t \right>$ of polynomials from $A_S$ converges to $0$ if and only if every choice of nonzero monomial summands $m_t$ from $p_t$ gives a sequence of monomials $\left< m_t \right>$ that converges to $0$. We can now define a topology $\mathcal{T}_S$ on $A_S$ by saying that $C \subset A_S$ is \emph{closed} with respect to $\mathcal{T}_S$ if and only if the limit of every sequence in $C$ convergent in $A_S$ is in $C$. 

We will consider the completion $(A_S)_{\widehat{\mathcal{T}}_S}$ of the ring $A_S$ with respect to the topology $\mathcal{T}_S$. Let $F$ denote the universal formal group law. We define an ideal by 
\begin{align*} I_S = \bigg(     u_{L_1} &+_F u_{L_2} - \left(\sum_{i=1}^m\right)_F u_{M_i} \text{ } \bigg| \text{ }L_1L_2 \cong \prod_{i=1}^m M_i \text{ and there is a } j \in \{1, \ldots, k\} \\ &\text{ s.t. } L_1,L_2 \in R_j \text{ and } M_i \in R_j \coprod \cdots \coprod R_k     \bigg).
\end{align*}
Note that the definitions of $A_S$, $\mathcal{T}_S$, and $I_S$ above depend on the group $G$, so when necessary for clarity we will denote these by $A_{G,S}$, $\mathcal{T}_{G,S}$ and $I_{G,S}$, respectively.

\begin{theorem}{(Abram and Kriz)} \label{ringcomp} For $G$ a finite abelian group, we have
\[(MU_G)_* \cong \varprojlim (A_S)_{\widehat{\mathcal{T}}_S}/I_S.
\]
\end{theorem}

This computation is illuminated somewhat by the following corollary of Theorem ~\ref{ringcomp} and the exposition of \cite{AKprep}, which corollary appears in \cite{Abram13}. The notation $[k]_F x$ denotes the $k$-fold sum $x +_F x +_F \cdots +_F x$ of $x$ under the nonequivariant universal formal group law.

\begin{cor} \label{zpnringcomp} Let $u_{[k]}$ denote $[p^k]_Fu$, and 
\[R_k = MU_*[u_j,u_j^{-1},b_j^{(i)} | i > 0, j \in \{1,2,\ldots,p^k-1\}][[u_{[k]}]]/([p^{n-k}]_Fu_{[k]}),\]
\[S_k = MU_*[u_j,u_j^{-1},b_j^{(i)} | i > 0, j \in \{1,2,\ldots,p^k-1\}][[u_{[k]}]]/([p^{n-k}]_Fu_{[k]})[u_{[k]}^{-1}],\]
\[R^n = MU_*[u_j,u_j^{-1},b_j^{(i)} | i >0, j \in \{1,2,\ldots,p^n-1\}]. \]
Then $(MU_{\mathbb{Z}/p^n})_*$ is the $n$-fold pullback of the diagram of rings
\begin{equation} \label{zpnringcompeq} \xymatrix{& & & & R^n \ar[d]^{\phi_{n-1}} \\
& & &  R_{n-1} \ar[r]^{\psi_{n-1}} \ar[d]^{\phi_{n-2}} & S_{n-1} \\
& & R_2 \ar[d]^{\phi_1} \ar[r]^\cdots & S_{n-2} & & \\
& R_1 \ar[d]^{\phi_0} \ar[r]^{\psi_1} & S_1 & & \\
R_0 \ar[r]^{\psi_0} & S_0. & & &\\
\\}
\end{equation}
The maps $\psi_k$ are localization by inverting $u_{[k]}$, and the maps $\phi_k$ are determined by the properties of sending $u_{[k+1]}$ to $[p]_F u_{[k]}$ and $b_j^{(i)}u_j$ to the coefficient of $x^i$ in $x +_F [j]_F u_{[k]}$. $\phi^{n-1}$ is determined by the property of sending $b_j^{(i)}u_j$ to the coefficient of $x^i$ in $x +_F [j]_F u_{[k]}$.
\end{cor}

\section{Equivariant formal group laws} \label{efgl}

 Let $G$ be a finite abelian group, and $E \in G \mathcal{S} \mathcal{U}$ a $G$-equivariant spectrum. Let $\widehat{G}$ denote the character group of $G$.  Cole, Greenlees, and Kriz \cite{CGK00} make the following definition.

\begin{definition} \label{efgl2} A $G$-equivariant formal group law over a commutative ring $k$ consists of:
\begin{enumerate}[(1)]
\item A $k$-algebra $R$ complete at an ideal $I$ and a cocommutative, coassociative, counital comultiplication 
\[ \Delta: R \rightarrow R \hat{ \otimes } R \text{   } \text{   } \text{   } (R \hat{ \otimes } R = (R \otimes R)_{\hat{I} \otimes R \oplus R \otimes I}),
\]
\item An $I$-continuous map of $k$-algebras $\epsilon : R \rightarrow  k^{G^*}$ compatible with comultiplication:
\[\xymatrix{
R \ar[d]_{\epsilon} \ar[r]^{\Delta} & R \hat{ \otimes } R \ar[d]^{\epsilon \hat{\otimes} \epsilon} \\
k^{G^*}  \ar[r]_{\psi \text{ } \text{ } \text{ } \text{ } \text{ } \text{ } \text{ } \text{ } \text{ }}  & k^{G^*} \otimes  k^{G^*} ,\\} \]
\item A system of elements $x_L \in R$, $L \in \hat{G}$ such that 
\[ x_L \text{ is regular for each } L \in \hat{G},
\]
\[R/(x_L) \cong E_* \text{ for each } L \in \hat{G},
\]
\[ I = \left(\prod_{L \in \hat{G}} x_L \right),
\]
and 
\[ x_L = (\epsilon(L) \otimes 1)\Delta(x_1) \text{  for  } L \in \hat{G}.
\]
\end{enumerate}
\end{definition}
One can adapt this definition to the case where $G$ is a compact Lie group by replacing the ideal $I$ above with the system of finite product ideals $(\prod_i x_{L_i})$ (cf. \cite{CGK00}). Note that Definition ~\ref{efgl2} gives an equivariant analog of a formal group scheme over $R$, which in general need not be smooth in the sense that it need not be representable by a formal power series algebra over $R$. For a discussion of the non-equivariant comparison between formal group laws and formal group schemes, see Section 37.3 of Hazewinkel's book \cite{Hazewinkel78} on formal groups.

Greenlees makes the following conjecture. 

\begin{conjecture}{(Greenlees, \cite{Greenlees01})} \label{Gconj} (Greenlees, \cite{Greenlees01}) For any complex oriented $G$-equivariant spectrum $E$ there is a unique homomorphism of rings $\theta : MU_G^* \rightarrow E^*$ such that $\theta$ induces maps that send the structures $(1)$, $(2)$, and $(3)$ for the canonical equivariant formal group law corresponding to $MU_G$ to the corresponding structure for $E$.
\end{conjecture}

In \cite[Example 11.3(i)]{CGK00} it is shown that  a complex orientation on a $G$-equivariant spectrum $E$ over the complete universe $\mathcal{U}$ specifies  a $G$-equivariant formal group law with $k = E^*$ and $R = E^* \mathbb{C}P^\infty_G$. Here $\mathbb{C}P_G^\infty$ is the complex projective space on the complete $G$-universe. Our goal is to follow the construction of \cite{CGK00} to compute the equivariant formal group law corresponding to the equivariant complex cobordism spectrum $MU_G$, for the case where $G$ is a finite abelian group. We can do this because $MU_G^*$ has a canonical complex orientation.

A partial result toward this conjecture has been proved by Greenlees.

\begin{theorem}{(Greenlees, \cite{Greenlees01})} For any finite abelian group $G$, let $L_G$ be the $G$-equivariant Lazard ring classifying $G$-equivariant formal group laws. The canonical homomorphism
\[
v: L_G \rightarrow MU_G^*
\]
is surjective and the kernel is Euler torsion, Euler-divisible and $\mathbb{Z}$-torsion.
\end{theorem}

We now proceed to describe the equivariant formal group law corresponding to $MU_G$ for $G$ a finite abelian group, with an eye toward Greenlees's Conjecture ~\ref{Gconj}.

\section{The equivariant formal group law of $MU_G$} \label{efglgeneralcase}

Our description of the equivariant formal group law of $MU_G$ is a consequence of Theorem ~\ref{ringcomp} and Definition ~\ref{efgl2}, and the description is similar in nature to that of Theorem ~\ref{ringcomp}. Before we can state our main theorem, we need a few definitions. For $S = \{ H_1 \subsetneq H_2 \subsetneq \cdots \subsetneq H_k\} \in P(G)$, $H_0 = \{e\}$, and $H_{k+1} = G$, define
\begin{align*} &Q_j = \prod_{L \in \overline{H_j^*}}(A_{H_j,S})_{\widehat{\mathcal{T}}_{H_j,S}}/I_{H_j,S}[[u_L | L \in R_j]]/(u_L +_F u_M = u_{LM})[[x_L]]; \\
&T_j = \prod_{L \in \overline{H_{j+1}^*}}(A_{H_j,S})_{\widehat{\mathcal{T}}_{H_j,S}}/I_{H_j,S}[[u_L | L \in R_j]]/(u_L +_F u_M = u_{LM})[u_L^{-1} | L \in R_j][[x_L]].
\end{align*}
$Q^k$ is defined similarly, also using the computations in the proof of \cite[Theorem 2]{AKprep}, as
\[
Q^k = (A_{G,S})_{\widehat{T}_{G,S}}/I_{G,S}[[x_L]].
\] 
Here $A^* = \text{Hom}(A,S^1)$ and $\overline{A} = A-\{0\}$. We then define $\mathcal{N}_S$ to be the diagram 
\begin{equation} \xymatrix{& & & & Q^k \ar[d]^{\phi^{k-1}} \\
& & &  Q_{k-1} \ar[r]^{\psi_{k-1}} \ar[d]^{\phi_{k-2}} & T_{k-1} \\
& & Q_2 \ar[d]^{\phi_1} \ar[r]^\cdots & T_{k-2} & & \\
& Q_1 \ar[d]^{\phi_0} \ar[r]^{\psi_1} & T_1 & & \\
Q_0 \ar[r]^{\psi_0} & T_0, & & &\\
\\}
\end{equation}
where the horizontal maps are given by localization by inverting Euler classes and the condition
\begin{equation} x_L \mapsto \prod_{M \equiv L (\bmod H_j)} x_M +_F (u_L - u_M),
\end{equation}
and the vertical maps are determined by sending $u_L^{(i)}$ to the coefficient of $x^i$ in $x +_F u_L$ and by sending
\begin{equation} x_L \mapsto x_L.
\end{equation}

Define another diagram $\widetilde{\mathcal{N}}_S$ as 
\begin{equation} \xymatrix{& & & & \widetilde{Q}^k \ar[d]^{\widetilde{\phi}^{k-1}} \\
& & &  \widetilde{Q}_{k-1} \ar[r]^{\widetilde{\psi}_{k-1}} \ar[d]^{\widetilde{\phi}_{k-2}} & \widetilde{T}_{k-1} \\
& & \widetilde{Q}_2 \ar[d]^{\widetilde{\phi}_1} \ar[r]^\cdots & \widetilde{T}_{k-2} & & \\
& \widetilde{Q}_1 \ar[d]^{\widetilde{\phi}_0} \ar[r]^{\widetilde{\psi}_1} & \widetilde{T}_1 & & \\
\widetilde{Q}_0 \ar[r]^{\widetilde{\psi}_0} & \widetilde{T}_0, & & &\\
\\}
\end{equation}
where 
\begin{align*} &\widetilde{Q}_j = \prod_{L \in \overline{H_j^*}}(A_{H_j,S})_{\widehat{\mathcal{T}}_{H_j,S}}/I_{H_j,S}[[u_L | L \in R_j]]/(u_L +_F u_M = u_{LM})[[y_L,z_L]]; \\
&\widetilde{T}_j = \prod_{L \in \overline{H_{j+1}^*}}(A_{H_j,S})_{\widehat{\mathcal{T}}_{H_j,S}}/I_{H_j,S}[[u_L | L \in R_j]]/(u_L +_F u_M = u_{LM})[u_L^{-1} | L \in R_j][[y_L,z_L]],
\end{align*}
and 
\[\widetilde{Q}^k = Q^k = (A_{G,S})_{\widehat{T}_{G,S}}/I_{G,S}[[y_L, z_L]].
\]
The maps $\widetilde{\phi}_j$ and $\widetilde{\psi}_j$ are defined as were $\phi_j$, and $\psi_j$, with the analogous conditions that
\begin{equation} \widetilde{\psi}_j (y_L) = \prod_{M \equiv L (\bmod H_j)} y_M +_F (u_L - u_M),
\end{equation}
\begin{equation} \widetilde{\psi}_j (z_L) = \prod_{M \equiv L (\bmod H_j)} z_M +_F (u_L - u_M),
\end{equation}
\begin{equation} \widetilde{\phi}_j(y_L) = y_L, \text{ and } \widetilde{\phi}_j(z_L) = z_L.
\end{equation}
Let $\mathbb{C}P^\infty_G = \mathbb{C}P(\mathcal{U})$ denote the complex projective space on the complete $G$-universe. We are now ready to state our main theorem.

\begin{theorem} \label{maintheorem} The equivariant formal group law of $MU_G$ consists of the following structures:
\begin{enumerate}[(a)]
\item the commutative ring $k = MU_G^*$, whose algebraic description is given by Theorem ~\ref{ringcomp}.
\item the $k$-algebra $R = \text{ho} \varprojlim \mathcal{N}_S$;
\item the ideal $I = (\prod_{L \in \widehat{G}} x_L)$;
\item the $k$-algebra $R \widehat{\otimes} R = \text{ho} \varprojlim \widetilde{\mathcal{N}}_S$;
\item the coproduct $\Delta: R \rightarrow R \widehat{\otimes} R$ is determined by maps $\Delta_S: \mathcal{N}_S \rightarrow \widetilde{\mathcal{N}}_S$ of diagrams, which send $Q_j \rightarrow \widetilde{Q}_j$, $T_j \rightarrow \widetilde{T}_j$, and are determined by the identity maps away from power series variables and by the conditions
\begin{equation*} x_L \mapsto \prod_{MN \equiv L (\bmod H_j)} (y_M +_F z_N);
\end{equation*}
\item the map $\epsilon: R \rightarrow (MU_G^*)^{G^*}$ is defined by choosing a basepoint $*_L$ in each connected component of 
\[ (\mathbb{C}P^\infty_G)^G = \coprod_{L \in \widehat{G}} \mathbb{C}P^\infty,
\]
where the superscript $G$ denotes fixed points.
$\epsilon$ is the induced map in cohomology of the $G$-equivariant map
\[\coprod_{L \in \widehat{G}} *_L \rightarrow G.
\]
\end{enumerate}
\end{theorem}

\begin{proof} This is accomplished by reconciling \cite{CGK00} with Theorem ~\ref{ringcomp}. Since $MU_G^*$ is complex stable and complex oriented, we can take $k = MU_G^*$ and
\begin{equation} \label{RforGfinab}  R = MU_G^* \mathbb{C}P_G^\infty.
\end{equation}
 The elements $x_L \in \widetilde{MU}_G^2 T(\gamma_G \otimes L)$ are Thom classes, computed just as in \cite[Section 4]{CGK00}, where $\gamma_G$ is the canonical line bundle on $\mathbb{C}P_G^\infty$ and $T$ denotes the Thom space. Let $x_0 \in \widetilde{MU}_G^2 T(\gamma_G)$ be the orientation class. Now let $\phi: \mathbb{C}P^\infty_G \rightarrow \mathbb{C}P^\infty_G$ classify $\gamma_G \otimes L$, i.e. $\phi^*(\gamma_G) = \gamma_G \otimes L$. Then we define
\begin{equation} x_L = \text{Im} (\widetilde{MU}_G^2T(\gamma_G) \rightarrow^{\widetilde{MU}_G^2T\phi} \widetilde{MU}_G^2T(\gamma_G \otimes L)).
\end{equation}
Let $U = \bigotimes_{L \in \widehat{G}} L$.
The ideal $I$ is 
\begin{equation} I = \left(\prod_{L \in \widehat{G}} x_L\right),
\end{equation}
where the product on the right is computed by the Thom diagonal
\[\Delta_t: T(\gamma_G \otimes U) \rightarrow \bigwedge_{L \in \widehat{G}} T (\gamma_G \otimes L),
\] 
as proposed in Theorem ~\ref{maintheorem}(b).

It follows from the Splitting Theorem \cite[Theorm 4.3]{CGK00} of Cole that $R = MU_G^* \mathbb{C}P^\infty_G$ is complete at $I$, and 
\begin{equation}R \widehat{\otimes} R \cong MU_G^*(\mathbb{C}P^\infty_G \times \mathbb{C}P^\infty_G).
\end{equation}
The comultiplication $\Delta$ is induced by the map classifying the tensor multiplication of line bundles:
\[\mu: \mathbb{C}P^\infty_G \times \mathbb{C}P^\infty_G \rightarrow \mathbb{C}P^\infty_G,
\]
i.e. for line bundles $\epsilon = f^* \gamma_G$, $\omega = g^* \gamma_G$, $\epsilon \otimes \omega = (\mu(f \times g))^*$. Choosing basepoints $*_L$ in each connected component of $(\mathbb{C}P^\infty_G)^G$, the description of the map $\epsilon$ given as Theorem ~\ref{maintheorem}(e) follows from \cite{CGK00}.

The above is documented in the note \cite{Kriz} of Kriz. Our goal now is to understand better the algebraic structure of the ring $MU_G^*\mathbb{C}P^\infty_G$. By Theorem 4.3 of \cite{CGK00}, we have
\begin{equation} \label{RforGeq}
MU_G^*\mathbb{C}P^\infty_G \cong MU_G^*\{\{x_0,x_{L_1},x_{L_1L_2}, \ldots \}\},
\end{equation}
where $L_1 \oplus L_2 \oplus \cdots$ is any splitting of the complete $G$-universe $\mathcal{U}$. Thus $x_0,x_{L_1},x_{L_1L_2}, \ldots$ are a flag basis of the complete universe $\mathcal{U}$, and $MU_G^*\{\{x_0,x_{L_1},x_{L_1L_2}, \ldots \}\}$ denotes 
\begin{equation} \left\{   \sum_{i=0}^\infty a_ix_0x_{L_1} \cdots x_{L_i} \bigg| a_i \in MU_G^* \right\}.
\end{equation}
We define 
\begin{equation} x_{L_1 \oplus L_2 \oplus \cdots \oplus L_m} = \prod x_{L_i},
\end{equation}
and now the right hand side of (\ref{RforGeq}) is well-defined. 

Now the splitting map $MU \rightarrow MU_G$ induces an isomorphism
\begin{equation} \pi_*^G(F(EG_+,MU)) \cong \pi_*^G(F(EG_+,MU_G)),
\end{equation}
and it follows that 
\begin{equation} \pi_*^G(F(EG_+,MU_G)^* \mathbb{C}P_G^\infty) \cong MU_*[[u_L| L \in \overline{G^*}]]/(u_L +_F u_M = u_{LM})[[x]].
\end{equation}
We are now able to give a better description of the elements $x_L$. Clearly,
\begin{equation} x_0 = x \in MU_*[[u_L| L \in \overline{G^*}]]/(u_L +_F u_M = u_{LM})[[x]],
\end{equation}
while 
\begin{equation} x_L = x_0 +_F u_L.
\end{equation}
Greenlees \cite[Theorem 11.2]{Greenlees01} gives
\begin{equation} \Phi^{G}MU_G^*\mathbb{C}P_G^\infty \cong \prod_{L \in \overline{G^*}} \Phi^G MU_G^*[[x_L]] = \prod_{L \in \overline{G^*}} \Phi^GMU_G^*[[x +_F u_L]].
\end{equation}

Now the description of $R$ given as Theorem ~\ref{maintheorem}(a) is a formal consequence of the above. Theorem ~\ref{maintheorem}(b) concerning $R \widehat{\otimes} R$ is obtained by a direct computation, and the description of the coproduct $\Delta$ in Theorem ~\ref{maintheorem}(c) follows from the definition of $\Delta$ above and from the work of \cite{AKprep} as summarized in Section ~\ref{cobordism}. Conditions (1)-(3) of Definition ~\ref{efgl2} are guaranteed by \cite{CGK00} and \cite{Kriz}, since the definitions of the algebras, maps, and elements of the equivariant formal group law given here are direct consequences of the example of complex stable, complex oriented cohomology theories and equivariant formal group laws as given in those papers.
\end{proof}

\section{The Case $G = \mathbb{Z}/p^n$} \label{zpnefgl}

There is intricate structure hiding beneath the surface of our description of the equivariant formal group law for $MU_G$ in the previous section. To illuminate some of this hidden structure, we present here the description of $MU_{\mathbb{Z}/p^n}$ as given by Theorem ~\ref{maintheorem}. 

We give a useful description of the elements $x_j$ arising from the diagram (\ref{zpnringcompeq}). Of course $x_0 = x \in MU_*[[u]]/([p^n]_Fu)[[x]]$, and $x_j = x_0 +_F [j]_Fu$. Let $R_k$, $S_k$, $0 \leq k \leq n-1$, and $R^n$ be as in Corollary ~\ref{zpnringcomp}, and refer to that result for notation. Then the element $u_j b_j^{(i)}$ of $R^n$ maps to an element of $S_{n-1}$ that does not include the term $u_{[n-1]}^{-1}$, so this element really lives in $R_{n-1}$. For $0 < k < n$, the resulting element of $R_k$ maps to an element of $S_{k-1}$ which does not include the term $u_{[k-1]}^{-1}$, so it really lives in $R_{k-1}$. This allows us to map the elements $u_j b_j^{(i)}$ of $R^n$ to $R_0 = MU_*[[u]]/([p^n]_Fu)$; call this map $\phi$. Then there is an implied map $\phi: MU_*[u_jb_j^{(i)} | i > 0, 1 \leq j \leq p^n -1 ][[x]] \rightarrow MU_*[[u]]/([p^n]_Fu)[[x]]$. Since $u_jb_j^{(i)}$ maps to the coefficient of $x^i$ in $x +_F [j]_Fu$, $x_j$ is the image under $\phi$ of the element
\[\sum_{i=0}^\infty u_jb_j^{(i)} x^i.
\]
We would also like a nice description of the $MU_{\mathbb{Z}/p^n}^*$-algebra
\[R = MU_{\mathbb{Z}/p^n}^* (\mathbb{C}P^\infty_{\mathbb{Z}/p^n}) = MU_{\mathbb{Z}/p^n}^* \{\{ \mathcal{U} \} \}
\]
as a product. Greenlees \cite[Theorem 11.2]{Greenlees01} gives us the following:

\begin{equation} \Phi^{\mathbb{Z}/p^n} MU_{\mathbb{Z}/p^n}^*  \{ \{ \mathcal{U} \} \} = \prod_{j=0}^{p^n -1} \Phi^{\mathbb{Z}/p^n}MU_{\mathbb{Z}/p^n}^* [[x_j]] = \prod_{j=0}^{p^n -1} \Phi^{\mathbb{Z}/p^n}MU_{\mathbb{Z}/p^n}^* [[x +_F [j]_F u]].
\end{equation}

Moreover, we obtain $R$ as an $n$-fold pullback, using Corollary ~\ref{zpnringcomp}. The various powers of the Euler class which are invertible on the diagram (\ref{zpnringcompeq}) allow for certain product decompositions of the ring $R = MU_{\mathbb{Z}/p^n}^*\mathbb{C}P^\infty_{\mathbb{Z}/p^n}$. Let $R^n$, $S_k$, $R_k$ stand for the cohomology rings now, rather than homology. Then $R$ is the pullback of the following diagram of rings:

\begin{equation} \label{ZpnRdiagram} \xymatrix{& &  \prod_{k \in (\mathbb{Z}/p^n)^*} R^n[[x_k]] \ar[d]^{\phi^{n-1}} \\
& \prod_{k \in (\mathbb{Z}/p^{n-1})^*} R_{n-1}[[x_k]] \ar[r]^{\psi_{n-1}} \ar[dr]^{\cdots} & \prod_{k \in (\mathbb{Z}/p^n)^*} S_{n-1}[[x_k]] & \\
& \prod_{k \in (\mathbb{Z}/p)^*}R_1[[x_k]] \ar[d]^{\phi_0} \ar[r]^{\psi_1} & \prod_{k \in (\mathbb{Z}/p^2)^*}S_1[[x_k]]  \\
R_0 \ar[r]^{\psi_0 \text{ } \text{ } \text{ }}[[x]] & \prod_{k \in (\mathbb{Z}/p)^*}S_0[[x_k]]. & \\
\\}
\end{equation}
The horizontal maps, as implied, are induced by the maps $\psi_k$ and the condition $x_j \mapsto \prod_{r \equiv j (\bmod ~p^k)} (x_r +_F [j-r]_Fu_{[k]})$. The vertical maps are induced by the maps $\phi_k$ and the condition $x_j \mapsto x_j$ for all $j$.

There is a similar description of $R \hat{\otimes} R$ as a pullback:

\begin{equation} \label{ZpnRRdiagram} \xymatrix{& &  \prod_{k,r \in (\mathbb{Z}/p^n)^*} R^n[[y_k,z_r]] \ar[d]^{\phi^{n-1}} \\
&  \prod_{k,r \in (\mathbb{Z}/p^{n-1})^*} R_{n-1}[[y_k,z_r]] \ar[r]^{\psi_{n-1}} \ar[dr]^{\cdots} & \prod_{k,r \in (\mathbb{Z}/p^n)^*} S_{n-1}[[y_k,z_r]] \\
& \prod_{k,r \in (\mathbb{Z}/p)^*}R_1[[y_k,z_r]] \ar[d]^{\phi_0} \ar[r]^{\psi_1} & \prod_{k,r \in (\mathbb{Z}/p^2)^*}S_1[[y_k,z_r]]   \\
R_0 \ar[r]^{\psi_0 \text{ } \text{ } \text{ }}[[y,z]] & \prod_{k,r \in (\mathbb{Z}/p)^*}S_0[[y_k,z_r]]. & \\
\\}
\end{equation}

The maps are determined by the maps of (\ref{ZpnRdiagram}) and the corresponding conditions for $y_k$ and $z_r$. Namely, under the horizontal maps, $y_k \mapsto \prod_{s \equiv k (\bmod ~p^j)} (y_l +_F [k-l]_F u_{[j]})$ $z_r \mapsto \prod_{s \equiv r (\bmod ~p^j)} (z_s +_F [r-s]_F u_{[j]})$. Under the vertical maps, $y_k \mapsto y_k$ and $z_r \mapsto z_r$.

We now specify the coproduct $\Delta: R \rightarrow R \hat{\otimes} R$ on the terms of the diagrams (\ref{ZpnRdiagram}) and (\ref{ZpnRRdiagram}). The map $\prod_{k \in (\mathbb{Z}/p^j)^*} R_j [[x_k]] \rightarrow \prod_{k,r \in (\mathbb{Z}/p^j)^*} R_j [[y_k,z_r]]$ is determined by the identity map on $R_j$ and the condition $x_k \mapsto \prod_{k_1 + k_2 = k}( y_{k_1} +_F z_{k_2})$, where by $k_1 + k_2 = k$ we of course mean $k_1 + k_2 \equiv k (\mod p^j)$. The map on the top right of the diagrams is defined similarly. The map  \[\prod_{k \in (\mathbb{Z}/p^j)^*} S_{j-1} [[x_k]] \rightarrow \prod_{k,r \in (\mathbb{Z}/p^j)^*} S_{j-1} [[y_k,z_r]]\]  is determined by the identity map on $S_{j-1}$ and the condition \[x_k \mapsto \prod_{k_1 + k_2 = k} y_{k_1} +_F z_{k_2}.\] Having nothing to add to the description of the map $\epsilon$ as given for general finite abelian groups $G$, this completes our description of the equivariant formal group law corresponding to $MU_{\mathbb{Z}/p^n}$.

\section{Some Future Directions}

There are several possible avenues for future work. The only case of Greenlees's Conjecture ~\ref{Gconj} yet proved is the case $G = \{e\}$, which was proved by Quillen in \cite{Quillen69}. Thus, a natural project would be to combine the special case of Theorem ~\ref{maintheorem} when $G=\mathbb{Z}/2$ with Strickland's computation of $MU_{\mathbb{Z}/2}^*$ given in \cite{Strickland01}. This could provide a more explicit picture of the equivariant formal group law associated with $MU_{\mathbb{Z}/2}$, but to prove Greenlees's Conjecture even in this case would require further insights on the Lazard ring $L_{\mathbb{Z}/2}$.

One could also use the computation of \cite{AKprep} to find generators and relations for $MU_G^*$, analogous to Strickland's computation of $MU_{\mathbb{Z}/2}$, for other groups $G$, for instance $G = \mathbb{Z}/3$ or $G= \mathbb{Z}/4$. This could then be used to further elucidate the description of the equivariant formal group law of $MU_G$ given here. 

Again, a better understanding of the Lazard ring $L_G$ is needed, and for this one may need to consider other perspectives on equivariant formal group laws in general, analogous to the various perspectives that are well-understood for non-equivariant formal group laws. 

%
%
%

\end{document}